\documentclass[twoside]{article}
\usepackage{latexsym}
\usepackage{pstricks}
\usepackage{amssymb}
\usepackage{amsmath}
\usepackage{amsfonts}
\usepackage{amsthm,array}
\usepackage{array}
\usepackage{epsfig}
\usepackage{graphicx}
\usepackage[numbers,sort&compress]{natbib}

\newtheorem{theorem}{Theorem}[section]

\newtheorem{lemma}{Lemma}[section]
\newtheorem{corollary}{Corollary}[section]

\newtheorem{definition}{Definition}[section]

\newtheorem{remark}{Remark}[section]

\newtheorem{remark-definition}{Remark and Definition}[section]
\newtheorem{rem-not}{Remark and Notation}[section]

\begin{document}
\title{\bf Positive Definiteness and Stability of Interval Tensors\thanks{The first author's work was supported by the National Natural Science Foundation of P.R. China (Grant No.12171064)}}
\date{}
\author{ Li Ye, Yisheng Song\thanks{Corresponding author E-mail: yisheng.song@cqnu.edu.cn}\\
	School of Mathematical Sciences,  Chongqing Normal University, \\
	Chongqing, 401331, P.R. China. \\ Email: neutrino1998@126.com (Ye); yisheng.song@cqnu.edu.cn (Song)}
\maketitle

{\noindent\bf Abstract.} In this paper, we focus on the positive definiteness and Hurwitz stability of interval tensors. First, we introduce auxiliary tensors $\mathcal{A}^z$ and establish  equivalent conditions for the positive (semi-)definiteness of interval tensors. That is, an interval tensor is positive definite if and only if all $\mathcal{A}^z$ are positive (semi-)definite. For Hurwitz stability, it is revealed that the stability of the symmetric interval tensor $\mathcal{A}_s^I$ can deduce the stability of the interval tensor $\mathcal{A}^I$, and the stability of symmetric interval tensors is equivalent to that of auxiliary tensors $\tilde{\mathcal{A}}^z$. Finally, taking $4$th order $3$-dimensional interval tensors as examples, the specific sufficient conditions are  built for their positive (semi-)definiteness.

\vspace{.3cm}
{\noindent\bf Mathematics Subject Classification.} 15A69, 90C23, 15A72, 15A63, 90C20, 90C30
\vspace{.3cm}

{\noindent\bf Keywords.} Positive definite, Interval tensor, Hurwitz stability, Fourth order tensors.

\section{Introduction}
\setcounter{section}{1}

In scientific and engineering computations, the uncertainty analysis has always been a core issue. Interval analysis as an effective tool for handling uncertainty, has yielded fruitful results in the theory of interval matrices since its introduction by Moore in the 1960s \cite{a1,a2,a3,a4,a5,a6}. By representing matrix elements as intervals, interval matrices provide a mathematical framework for characterizing uncertainties in system parameters, playing a key role in many areas such as stability analysis of dynamic systems and engineering error estimation \cite{b1,b2,b3,b4}. However, with the advancement of scientific research, many practical problems have extended from linear systems to multilinear systems. Tensors as higher-dimensional generalizations of matrices, have demonstrated strong descriptive capabilities in many fields such as data science, quantum physics, and complex networks \cite{c1,c2,c3,c4,c5,c6,c7,c8,c9,c10,c11}. The interval tensor theory is not only a natural extension of interval matrices analysis theory, but also an urgent need for addressing high dimensional uncertainty problems.

Positive definiteness and stability are important concepts in linear algebra and dynamical systems theory, there have been numerous relevant studies \cite{d1,d2,d3,d4,d5,rohn1994}. 
Rohn \cite{rohn1994} have systematically established criteria for determining positive definiteness, Hurwitz stability, and Schur stability of interval matrices, and proved that an interval matrix possesses a certain property if and only if a finite number of its vertex matrices satisfy that property.

It is well-known that the positive definiteness of tensors is equivalent to all H-eigenvalues (or Z-eigenvalues) being positive (non-negative) \cite{qi2005}.  The sign of minimum eigenvalue may verify the positive definiteness of tensors, and moreover, it may check the stability of multi-linear systems \cite{qi2005,lim2005,e1,e2,e3,e4,e5,e6,e7,e8,ql2017}. For interval tensors, Beheshti et al. \cite{bf2022} extended some classes of interval matrices to classes of interval tensors, and studied some properties and characterizations of them. Cui and Zhang \cite{CZ2023} shown the bounds of H-eigenvalues of an even order real symmetric interval tensor or a nonnegative interval tensor. Rahmati and Tawhid \cite{sr2020} generalize the row-property of a set of matrices to the slice-property of a set of tensors, pointing out that the slice-positive definiteness of a convex tensor set is equivalent to that of its extreme point set (Lemma \ref{lep}). However, for $m$th order $n$-dimensional interval tensors, which is a type of convex tensor set, the number of extreme points is $2^{n^m}$, leading to exponential growth of computational complexity. In terms of Hurwitz stability, Bozorgmanesh et al. \cite{boz2020} have established an equivalence relationship between the stability of symmetric interval tensors and the positive definiteness of negative interval tensors, and provided an estimation range for the real parts of E-eigenvalues by characterizing the counterpart of Bendixson's theorem in tensors, which is analogous to its formulation in matrices, this conclusion only applies to symmetric interval tensors. Other studies on interval tensors can be found in \cite{f1,f2,f3}. Additionally, the relationship between the stability of asymmetric interval tensors and their real symmetricized tensors has not been fully clarified, leaving a theoretical gap.

To address the above issues, based on the existing theory of positive definiteness and stability of interval tensors and combined with specific example analysis, this paper conducts the following core work. To begin with, introduce the auxiliary index set $Y=\{z\in\mathbb{R}^n\mid|z_j|=1\}$, define the specific tensor $\mathcal{A}^z=\mathcal{A}^c-\Delta\times_1 T_z\times_2 T_z\times\cdots\times_m T_z$ (where $T_z$ is a diagonal matrix), and prove that when $m$ is even, the positive definiteness of the interval tensor $\mathcal{A}^I$ can be determined by verifying only $2^{n-1}$ tensors, reducing the computational complexity from $2^{n^m}$ to $2^{n-1}$. In addition, establish the equivalence relationship between the positive definiteness of the interval tensor $\mathcal{A}^I$ and its symmetric interval tensor $\mathcal{A}_s^I$, and prove that if $\mathcal{A}_s^I$ is Hurwitz stable, then $\mathcal{A}^I$ must be Hurwitz stable. Finally, taking $4$th order $3$-dimensional interval tensors as the research object, multiple sets of directly applicable sufficient conditions for positive definiteness are obtained, providing specific tools for the performance analysis of low-order and low-dimensional uncertain systems in engineering.

\section{Preliminaries}
\setcounter{section}{2}

Within this section, we present the necessary definitions and notations related to tensors and interval tensors.

Herein, we will establish the following notation conventions. Vectors will be denoted as $\{x, y, \cdots \}$, matrices will be represented by $\{A, B, \cdots \}$, and tensors will be signified as $\{\mathcal{A}, \mathcal{B}, \cdots \}$. Denote the set of all $m$th order $n$-dimensional real tensors as $T_{m,n}$, and the entries of an $m$th order $n$-dimensional tensor $\mathcal{A}$ are denoted by $(a_{i_1i_2\cdots i_m})$, where $i_j\in[n]=\{1,2,\cdots,n\}$. And entries $a_{ii\cdots i}$ for $i\in[n]$ are called its diagonal entries, and others are called off-diagonal entries. $x\leq(<) y$, $A\leq(<) B$ and $\mathcal{A}\leq(<) \mathcal{B}$ are to be understood componentwise, and we introduce its absolute value as the tensor $|\mathcal{A}|=(|a_{i_1i_2\cdots i_m}|)$.

For any tensor $\mathcal{A} = (a_{i_1i_2\cdots i_m})\in T_{m,n}$, if its entries $a_{i_1i_2\cdots i_m}$ remain unchanged under any permutation of their indices, i.e., for every $\sigma \in S_m$ , where $S_m$ denotes the symmetric group of $m$ elements,
 $$a_{i_1 i_2 \cdots i_m} = a_{\sigma(i_1) \sigma(i_2) \cdots \sigma(i_m)},$$
then $\mathcal{A}$ is termed a symmetric tensor. The set of all $m$th-order $n$-dimensional real symmetric tensors is denoted as $S_{m,n}$.

For an $m$th degree homogeneous polynomial of $n$ variables,
$$f_{\mathcal{A}}(x)=\sum_{i_1\cdots i_m=1}^na_{i_1\cdots i_m}x_{i_1}\cdots x_{i_m}={\mathcal{A}} x^m,$$
where its coefficient tensor $\mathcal{A} = (a_{i_1i_2\cdots i_m})\in T_{m,n}$, there is unique symmetric tensor $\mathcal{A}_s = ((a_s)_{i_1i_2\cdots i_m})\in S_{m,n}$ such that
$$f_{\mathcal{A}}(x)\equiv\sum_{i_1\cdots i_m=1}^n(a_s)_{i_1\cdots i_m}x_{i_1}\cdots x_{i_m},$$
which is the symmetrization of $\mathcal{A}$, and
$$(a_s)_{i_1i_2\cdots i_m}=\frac{1}{m!}\sum_{\sigma \in S_m} a_{\sigma(i_1)\cdots \sigma(i_m)}$$

Let $\underline{\mathcal{A}}=(\underline{a}_{i_1i_2\cdots i_m})$, $\overline{\mathcal{A}}=(\overline{a}_{i_1i_2\cdots i_m})$ be $m$th order $n$-dimensional real tensors, and $\underline{\mathcal{A}}\leq\overline{\mathcal{A}}$. The set of tensors
$$\mathcal{A}^I=[\underline{\mathcal{A}},\overline{\mathcal{A}}]= \{\mathcal{A}:\underline{\mathcal{A}}\leq \mathcal{A}\leq \overline{\mathcal{A}}\}$$
is called an interval tensor.

Denote $\mathcal{A}^c=(a^c_{i_1i_2\cdots i_m})=\frac{\underline{\mathcal{A}}+\overline{\mathcal{A}}}{2}$ and $\Delta=(\delta_{i_1i_2\cdots i_m})=\frac{\overline{\mathcal{A}}-\underline{\mathcal{A}}}{2}$, then $\mathcal{A}^I=[\mathcal{A}^c- \Delta, \mathcal{A}^c+ \Delta]$. Clearly, $\Delta\geq\mathcal{O}$ is a nonnegative tensor, i.e., all of its entries $\delta_{i_1i_2\cdots i_m}$ are nonnegative, where $\mathcal{O}$ is the tensor in $T_{m,n}$ with all its entries being zero. $\mathcal{A}^I$ is said to be symmetric if both $\mathcal{A}^c$ and $\Delta$ are symmetric. With each interval tensor $\mathcal{A}^I=[\mathcal{A}^c- \Delta, \mathcal{A}^c+ \Delta]$ we shall associate the symmetric interval matrix
$$\mathcal{A}^I_s=[\mathcal{A}_s^c- \Delta_s, \mathcal{A}_s^c+ \Delta_s].$$

\begin{remark}
If $\mathcal{A}\in \mathcal{A}^I$, then $\mathcal{A}_s\in \mathcal{A}^I_s$ and $\mathcal{A}^I$ is symmetric if and only if
$\mathcal{A}^I=\mathcal{A}^I_s$.
\end{remark}
For $\mathcal{A}= (a_{i_1\cdots i_m})\in T_{m,n}$ and $x\in\mathbb{R}^n$, $P=(p_{ij})\in\mathbb{R}^{n\times n}$, and $k\in[m]$, $\mathcal{A}\times_1 P\times_2P\times_3\cdots\times_m P$ is a tensor in $T_{m,n}$ which defined by
$$(\mathcal{A}\times_1 P\times_2P\times_3\cdots\times_m P)_{i_1i_2\cdots i_m}=\sum\limits_{j_1j_2\cdots j_m=1}^n a_{j_1j_2\cdots j_m}p_{i_1j_1}\cdots p_{i_mj_m}.$$
And $\mathcal{A}x^{m-1}$ is a vector in $\mathbb{R}^n$ with the $i$th component
$$(\mathcal{A}x^{m-1})_i:=\sum\limits_{i_2,\cdots i_m=1}^n a_{ii_2\cdots i_m}x_{i_2}\cdots x_{i_m},$$
for $1\leq i \leq n$.

\begin{definition}\label{def1}\textup{\cite{qi2005,sr2020,boz2020}}
Let $m$ is even, ${\mathcal{A}},{\mathcal{A}^c,\Delta}\in{\mathcal{T}}_{m,n}$ and $\Delta\geq \mathcal{O}$, ${\mathcal{A}}^I=[\mathcal{A}^c- \Delta, \mathcal{A}^c+ \Delta]$.
\begin{itemize}
\item[(a)] $\mathcal{A}$ is called positive semi-definite if $\mathcal{A}x^m\geq0$ for each $x\in{\mathbb{R}}^n\setminus\{0\}$, $\mathcal{A}^I$ is called positive semi-definite if each $\mathcal{A}\in\mathcal{A}^I$ is positive semi-definite.
\item[(b)] $\mathcal{A}$ is called positive definite if $\mathcal{A}x^m>0$ for each $x\in{\mathbb{R}}^n\setminus\{0\}$, $\mathcal{A}^I$ is called positive definite if each $\mathcal{A}\in\mathcal{A}^I$ is positive definite.
\end{itemize}
\end{definition}

\begin{definition}\label{defh}\textup{\cite{qi2005}}
Let ${\mathcal{A}}\in{\mathcal{T}}_{m,n}$, $\lambda\in\mathbb{C}$ and $x=(x_1,x_2,\cdots,x_n)^\top\in\mathbb{C}^n$.
$\lambda$ is called an eigenvalue and $x$ an eigenvector corresponding to $\lambda$ if
\begin{equation}\label{1.1}
\mathcal{A}x^{m-1}=\lambda x^{[m-1]}
\end{equation}
where $x^{[m-1]}=(x_1^{m-1},x_2^{m-1},\cdots,x_n^{m-1})^\top$.
Moreover, when $\lambda\in\mathbb{R}$ and $x\in\mathbb{R}^n$, $\lambda$ is called an H-eigenvalue and $x$ an H-eigenvector corresponding to $\lambda$.
\end{definition}

\begin{definition}\label{defz}\textup{\cite{qi2005}}
Let ${\mathcal{A}}\in{\mathcal{T}}_{m,n}$, $\lambda\in\mathbb{C}$ and $x=(x_1,x_2,\cdots,x_n)^\top\in\mathbb{C}^n$.
$\lambda$ is called an E-eigenvalue and $x$ an E-eigenvector corresponding to $\lambda$ if
\begin{equation}\label{1.2}
\left
\{\begin{array}{ll}
\mathcal{A}x^{m-1}=\lambda x,\\
x^{\top}x=1.
\end{array}
\right.
\end{equation}
Moreover, when $\lambda\in\mathbb{R}$ and $x\in\mathbb{R}^n$, $\lambda$ is called and Z-eigenvalue and $x$ an Z-eigenvector corresponding to $\lambda$.
\end{definition}

We shall denote by $\lambda_{H~\min}(\mathcal{A})$ ($\lambda_{Z~\min}(\mathcal{A})$) and $\lambda_{H~\max}(\mathcal{A})$ ($\lambda_{Z~\max}(\mathcal{A})$), the minimum and maximum H(Z)-eigenvalue of $\mathcal{A}$, respectively (obviously, $\lambda_{H~\min}(-\mathcal{A})=-\lambda_{H~\max}(\mathcal{A})$, $\lambda_{Z~\min}(-\mathcal{A})=-\lambda_{Z~\max}(\mathcal{A})$). Denote the H-spectral radius by $\rho_H(\mathcal{A})$.

\begin{lemma}\label{HP}\textup{\cite{qi2005}}
Let $\mathcal{A}\in S_{m,n}$ and $m$ is even.
\begin{itemize}
  \item [(a)] $\mathcal{A}$ is positive (semi-)definite if and only if all of its H-eigenvalues are positive (nonnegative), which means $\lambda_{H~\min}(\mathcal{A})>(\geq)0$.
  \item [(b)] $\mathcal{A}$ is positive (semi-)definite if and only if all of its Z-eigenvalues are positive (nonnegative), which means $\lambda_{Z~\min}(\mathcal{A})>(\geq)0$.
\end{itemize}
\end{lemma}

\begin{lemma}\label{Cth}\textup{\cite{ql2017}}
Let $\mathcal{A}\in S_{m,n}$. Then the eigenvalues of of $\mathcal{A}$ lie in the union of $n$ disks in $\mathbb{C}$. These $n$ disks have the diagonal entries of $\mathcal{A}$ as their centers, and the sums of the absolute values of the off-diagonal entries as their radii.
\end{lemma}

\vspace{.3cm}

\section{Positive definiteness of interval tensors}
\setcounter{section}{3}

Rahmati and Tawhid showed that a compact convex set $\mathcal{C}$ of tensors has the slice-positive-(semi-)definite-property if
and only if the set $\mathcal{X}$ of extreme points of $\mathcal{C}$, has the slice-positive-(semi-)definite-property. $\mathcal{C}$ has the slice-positive-(semi-)definite-property if every tensor in $\hat{\mathcal{C}}$ is an positive-(semi-)definite-tensor, where slice-completion $\hat{\mathcal{C}}$ is the set of all tensors that have one slice in common with a tensor in $\mathcal{C}$ \cite{sr2020}.
A point $x \in\mathcal{C}$ is called an extreme point of $\mathcal{C}$ if whenever we have $x =(1-k)y + k z$, for some $y, z \in \mathcal{C}$ and $k \in (0, 1)$, we conclude $x = y = z$; see \cite{spm}. In the following we always assume $m$ is even, since this is only meaningful for even order tensors.

\begin{lemma}\label{lep}\textup{\cite{sr2020}}
Suppose that $\mathcal{C}$ is a compact, convex set in $T_{m,n}$. Let $\mathcal{X}$ be the set of all extreme points of $\mathcal{C}$. Then $\mathcal{C}$ has the slice-positive-(semi-)definite-property if and only if $\mathcal{X}$ has the slice-positive-(semi-)definite-property.
\end{lemma}

If $\mathcal{C}=\mathcal{A}^I=[\underline{\mathcal{A}},\overline{\mathcal{A}}]$, then the set $\mathcal{X}$ of all extreme points of $\mathcal{C}$ is $$\mathcal{E}=\{\mathcal{A}=(a_{i_1i_2\cdots i_m})\in T_{m,n}:a_{i_1i_2\cdots i_m}=\underline{a}_{i_1i_2\cdots i_m} \mbox{ or } a_{i_1i_2\cdots i_m}=\overline{a}_{i_1i_2\cdots i_m}, i_1,i_2,\cdots, i_m\in[n]\}.$$
Therefore, the following result can be obtained by Lemma \ref{lep}.

\begin{corollary}\label{cor1}
Let $\mathcal{A}^I=[\underline{\mathcal{A}},\overline{\mathcal{A}}]$ be an interval tensor in $T_{m,n}$. Then $\mathcal{E}$ has is positive (semi-)definite if and only if $\mathcal{C}$ is positive (semi-)definite.
\end{corollary}

Based on this result, verifying the positive definiteness of an interval tensor $\mathcal{A}^I$ in $T_{m,n}$ requires checking the positive definiteness of $2^{n^m}$ tensors as the cardinality of $\mathcal{E}$ is $2^{n^m}$. In what follows, we will improve this result, reducing the number of tensors that need verification to $2^{n-1}$.

First it need to define an auxiliary index set
$Y= \{z\in \mathbb{R}^n; |z_j|=1 \mbox{ for } j=1,\cdots,n\},$
which has a cardinality of $2^n$. For each $z\in Y$ we denote by $T_z$ as the $n\times n$ diagonal matrix whose diagonal entries from the vector $z$. Next, for each $z\in Y$ let us introduce the matrix $\mathcal{A}^z$ via the expression,
$$\mathcal{A}^z=(a^z_{i_1i_2\cdots i_m})=\mathcal{A}^c-\Delta\times_1 T_z\times_2 T_z\times_3\cdots\times_m T_z.$$
For each $i_j\in[n]$, $j=1,2,\cdots,m$, the entry $a^z_{i_1i_2\cdots i_m}$ of $\mathcal{A}^z$ satisfies,
 $$a^z_{i_1i_2\cdots i_m}=a^c_{i_1i_2\cdots i_m}- \delta_{i_1i_2\cdots i_m}\mbox{ if }z_{i_1}z_{i_2}\cdots z_{i_m}=1,$$
 $$a^z_{i_1i_2\cdots i_m}=a^c_{i_1i_2\cdots i_m}+ \delta_{i_1i_2\cdots i_m}\mbox{ if }z_{i_1}z_{i_2}\cdots z_{i_m}=-1.$$
 Thus, $\mathcal{A}^z$ belongs to $\mathcal{A}^I$ for every $z\in Y$. When $m$ is even, since $\mathcal{A}^{-z}=\mathcal{A}^z$, the number of distinct tensors $\mathcal{A}^z$ is at most $2^{n-1}$ (and equal to $2^{n-1}$ if $\Delta> 0$). If $\mathcal{A}^I$ is symmetric, each $\mathcal{A}^z$ is symmetric as well. These tensors $\mathcal{A}^z$, $z\in Y$, will be employed to characterize positive (semi-)definiteness of an interval tensor through finite procedures.

Let introduce a function $g: T_{m,n}\rightarrow \mathbb{R}$ defined for a tensor $\mathcal{A}\in T_{m,n}$ by
\begin{equation}\label{gx}
g(\mathcal{A})=\min\limits_{x\neq0}\frac{\mathcal{A}x^m}{\sum\limits_{i=1}x_i^{m}}.
\end{equation}
Obviously, $g$ is well defined. In the following theorem we sum up the basic properties of $g$ that will be used in the proofs of the main theorems in the subsequent sections.

\begin{theorem}\label{th1}
The function $g$ has the following properties.
\begin{itemize}
  \item [$(\romannumeral1)$] For each $\mathcal{A}\in T_{m,n}$, $g(\mathcal{A}) =g(\mathcal{A}_s)$.
  \item [$(\romannumeral2)$] For each $\mathcal{A}\in S_{m,n}$, $g( \mathcal{A} )=\lambda_{H~\min} (\mathcal{A})$.
  \item [$(\romannumeral3)$] For each $\mathcal{A},\mathcal{D}\in T_{m,n}$, $|g(\mathcal{A} + \mathcal{D}) -g(\mathcal{A})| \leq\rho_H(\mathcal{D}_s)$.
  \item [$(\romannumeral4)$] $g$ is continuous in $T_{m,n}$;
  \item [$(\romannumeral5)$] There is
  $$\min \{g(\mathcal{A});\mathcal{A}\in \mathcal{A}^I\}=\min \{g(\mathcal{A}^z); z\in Y\}$$
  for each interval tensor $\mathcal{A}^I$.
  \item [$(\romannumeral6)$] There is
  $$\min \{g(\mathcal{A});\mathcal{A}\in \mathcal{A}^I\}=\min \{g(\mathcal{A}); \mathcal{A}\in \mathcal{A}_s^I\}$$
  for each interval tensor $\mathcal{A}^I$.
\end{itemize}
\end{theorem}
\begin{proof} $(\romannumeral1)$ It can be obtained according to the definition of symmetrization.

$(\romannumeral2)$ See the proof of Theorem 2.2 in Ref. \cite{ql2017}.

$(\romannumeral3)$ For each $\mathcal{A}$, $\mathcal{D}\in T_{m,n}$, it can be obtained from (\ref{gx}) it follows
$$g(\mathcal{A}+\mathcal{D})\geq g(\mathcal{A})+g(\mathcal{D}),$$
which implies
\begin{align*}
g(\mathcal{A}) & =g((\mathcal{A}+\mathcal{D}) + (-\mathcal{D})) \\
   & \geq g(\mathcal{A}+\mathcal{D}) + g(-\mathcal{D}),
\end{align*}
then
\begin{align*}
|g(\mathcal{A}+\mathcal{D})-g(\mathcal{A})| &\leq \max \{|g(\mathcal{D})|, |g(-\mathcal{D})| \}\\
&=\max \{|g(\mathcal{D}_s)|, |g(-\mathcal{D}_s))|\}\\
&=\max \{|\lambda_{H~\min}(\mathcal{D}_s)|, |\lambda_{H~\max}(\mathcal{D}_s)|\}\\
&=\rho_H(\mathcal{D}_s).
\end{align*}

$(\romannumeral4)$ According $(\romannumeral3)$,
$$|g(\mathcal{A}+\mathcal{D})-g(\mathcal{A})| \leq\rho_H(\mathcal{D}_s)$$
for each $\mathcal{A}$, $\mathcal{D}\in T_{m,n}$. By Lemma \ref{Cth}, one can prove that $g$ is continuous in $T_{m,n}$.

$(\romannumeral5)$ One the one hand, for each interval tensor $\mathcal{A}^I$, clearly, as $\mathcal{A}^z\in\mathcal{A}^I$,
$$\min\{g(\mathcal{A}):\mathcal{A}\in\mathcal{A}^I\}\leq \min\{g(\mathcal{A}^z):z\in Y\}.$$

On the other hand, for each $\mathcal{A}\in \mathcal{A}^I$, $x=(x_1,x_2,\cdots,x_n)^{\top}\in\mathbb{R}^n$ and $i\in[n]$,
\begin{align*}
|x_i(\frac{\mathcal{A}x^{m-1}}{\sum\limits_{j=1}x_j^m})_i-x_i(\frac{\mathcal{A}^cx^{m-1}}{\sum\limits_{j=1}x_j^m})_i| &=|x_i(\frac{(\mathcal{A}-\mathcal{A}^c)x^{m-1}}{\sum\limits_{j=1}x_j^m})_i|\\
&\leq|x_i|(\frac{\Delta |x^{m-1}|}{\sum\limits_{j=1}x_j^m})_i.
\end{align*}
Let $z=\mbox{sgn}(x)=(\mbox{sgn}(x_1),\mbox{sgn}(x_2),\cdots,\mbox{sgn}(x_n))^{\top}$, where
$\mbox{sgn}(x_i)=\left
\{\begin{array}{ll}
1\mbox{ if }x_i\geq0,\\
-1\mbox{ if }x_i<0,
\end{array}
\right.$ then
\begin{align*}
x_i(\frac{\mathcal{A}x^{m-1}}{\sum\limits_{j=1}x_j^m})_i &\geq x_i(\frac{\mathcal{A}^cx^{m-1}}{\sum\limits_{j=1}x_j^m})_i-|x_i|(\frac{\Delta |x^{m-1}|}{\sum\limits_{j=1}x_j^m})_i\\
&=x_i\frac{\sum\limits_{i_2,\cdots,i_m=1}^na^c_{ii_2\cdots i_m}x_{i_2}\cdots x_{i_m}}{\sum\limits_{j=1}x_j^m}-|x_i|\frac{\sum\limits_{i_2,\cdots,i_m=1}^n\delta_{ii_2\cdots i_m}|x_{i_2}\cdots x_{i_m}|}{\sum\limits_{j=1}x_j^m}\\
&=\frac{\sum\limits_{i_2,\cdots,i_m=1}^na^c_{ii_2\cdots i_m}x_ix_{i_2}\cdots x_{i_m}}{\sum\limits_{j=1}x_j^m}-\frac{\sum\limits_{i_2,\cdots,i_m=1}^n\delta_{ii_2\cdots i_m}|x_ix_{i_2}\cdots x_{i_m}|}{\sum\limits_{j=1}x_j^m}\\
&=\frac{\sum\limits_{i_2,\cdots,i_m=1}^na^c_{ii_2\cdots i_m}x_ix_{i_2}\cdots x_{i_m}}{\sum\limits_{j=1}x_j^m}-\frac{\sum\limits_{i_2,\cdots,i_m=1}^n\delta_{ii_2\cdots i_m}\mbox{sgn}(x_ix_{i_2}\cdots x_{i_m})x_ix_{i_2}\cdots x_{i_m}}{\sum\limits_{j=1}x_j^m}\\
&=\frac{\sum\limits_{i_2,\cdots,i_m=1}^n[a^c_{ii_2\cdots i_m}-\delta_{ii_2\cdots i_m}\mbox{sgn}(x_ix_{i_2}\cdots x_{i_m})]x_ix_{i_2}\cdots x_{i_m}}{\sum\limits_{j=1}x_j^m}\\
&=\frac{\sum\limits_{i_2,\cdots,i_m=1}^na^z_{ii_2\cdots i_m}x_ix_{i_2}\cdots x_{i_m}}{\sum\limits_{j=1}x_j^m}\\
&=x_i(\frac{\mathcal{A}^z x^{m-1}}{\sum\limits_{j=1}x_j^m})_i.
\end{align*}

Hence,
$$\frac{\mathcal{A}x^m}{\sum\limits_{j=1}x_j^m}\geq\frac{\mathcal{A}^zx^m}{\sum\limits_{j=1}x_j^m}\geq\min\{g(\mathcal{A}^z);z\in Y\}$$
holds for each $\mathcal{A} \in \mathcal{A}^I$, which implies that
$$\min\{g(\mathcal{A}):\mathcal{A}\in\mathcal{A}^I\}\geq \min\{g(\mathcal{A}^z):z\in Y\},$$
the assertion follows.

$(\romannumeral6)$ For each $z \in Y$ denote by $\widehat{\mathcal{A}}^z=(\widehat{a}^z_{i_1i_2\cdots i_m})$ the tensor $\mathcal{A}^z$ for $ \mathcal{A}_s^I$, i.e.,
$$\widehat{\mathcal{A}}^z=\mathcal{A}_s^c-\Delta_s\times_1 T_z\times_2 T_z\times_3\cdots\times_m T_z,$$
where
\begin{align*}
\widehat{a}^z_{i_1i_2\cdots i_m}&=(a_s)^c_{i_1i_2\cdots i_m}-(\delta_s)_{i_1i_2\cdots i_m}\mbox{sgn}({z_{i_1}z_{i_2}\cdots z_{i_m}})\\
&=\frac{1}{m!}\sum_{\sigma \in S_m} a^c_{\sigma(i_1)\cdots \sigma(i_m)}-\frac{1}{m!}\sum_{\sigma \in S_m} \delta_{\sigma(i_1)\cdots \sigma(i_m)}\mbox{sgn}({z_{i_1}z_{i_2}\cdots z_{i_m}})\\
&=\frac{1}{m!}\sum_{\sigma \in S_m} [a^c_{\sigma(i_1)\cdots \sigma(i_m)}-\delta_{\sigma(i_1)\cdots \sigma(i_m)}\mbox{sgn}({z_{i_1}z_{i_2}\cdots z_{i_m}})]\\
&=\frac{1}{m!}\sum_{\sigma \in S_m} (a_s)^z_{\sigma(i_1)\cdots \sigma(i_m)}.
\end{align*}
Thus $\widehat{\mathcal{A}}^z=\mathcal{A}_s^z$ for each $z \in Y$, then employing $(\romannumeral1)$ we obtain
$$g(\mathcal{A}^z)=g(\mathcal{A}_s^z),$$
hence, the assertion $(\romannumeral5)$ implies that the minimum values of $g$ over $\mathcal{A}^I$ and $\mathcal{A}_s^I$ are equal.
\end{proof}

As a consequence of Theorem \ref{th1}, we obtain this characterization.

\begin{theorem}\label{th2}
Let $\mathcal{A}^I$ be an interval tensor in $T_{m,n}$. Then the following assertions are equivalent:
\begin{itemize}
  \item [$(a)$] $\mathcal{A}^I$ is positive (semi-)definite,
  \item [$(b)$] $\mathcal{A}^I_s$ is positive (semi-)definite,
  \item [$(c)$] $\mathcal{A}^z$ is positive (semi-)definite for each $z\in Y$.
\end{itemize}
\end{theorem}
\begin{proof} By Lemma \ref{HP}, $\mathcal{A}^I$ is positive (semi-)definite if and only if
$$\min \{g(\mathcal{A}): \mathcal{A}\in \mathcal{A}^I \} >(\geq) 0$$
holds. Then by $(\romannumeral6)$ of Theorem \ref{th1}, $(a)$ and $(b)$ are equivalent. And from the assertion $(\romannumeral5)$ of Theorem \ref{th1}, $(a)$ and $(c)$ are equivalent.
\end{proof}

\vspace{.3cm}

\section{Stability of interval tensors}
\setcounter{section}{4}
A tensor $\mathcal{A}\in T_{m,n}$ is called Hurwitz stable if the real part its any E-eigenvalue Re$\lambda < 0$, in other words, all its eigenvalues lie in the open left half of the complex plane. An interval tensor $\mathcal{A}^I$ in $T_{m,n}$ is said to be Hurwitz stable if each $\mathcal{A}\in \mathcal{A}^I$ is Hurwitz stable. Bozorgmanesh, et al. showed that, for every E-eigenvalue $\lambda$ of tensor $\mathcal{A}\in T_{m,n}$ satisfies
\begin{equation}\label{Bth}
\lambda_{Z~\min}(\mathcal{A}_s)\leq \mbox{Re} \lambda\leq\lambda_{Z~\max}(\mathcal{A}_s),
\end{equation}
which is analogous to the well-known Bendixson's theorem in the matrix case \cite{bth,boz2020}. 
And by utilizing this, the relationship between Hurwitz stability and positive definiteness of interval tensor has been obtained.
\begin{lemma}\label{sta}\textup{\cite{boz2020}}
Let $\mathcal{A}^I=[\underline{\mathcal{A}},\overline{\mathcal{A}}]$ be symmetric interval tensor in $T_{m,n}$. $\mathcal{A}^I$ is Hurwitz stable if and only
if $-\mathcal{A}^I=[-\overline{\mathcal{A}},-\underline{\mathcal{A}}]$ is positive definite.
\end{lemma}
Based on this lemma, similar to (c) of Theorem \ref{th2}, we can derive another equivalent condition for Hurwitz stability of interval tensors. This equivalent condition merely requires verifying the stability of a finite number of tensors.

Unlike the previous section, where we used the tensors $\mathcal{A}^z=\mathcal{A}^c-\Delta\times_1 T_z\times_2 T_z\times_3\cdots\times_m T_z$, in the present section, we will characterize Hurwitz stability by means of tensors
$$\widetilde{\mathcal{A}}^z=(\widetilde{a}^z_{i_1i_2\cdots i_m})=\mathcal{A}^c+\Delta\times_1 T_z\times_2 T_z\times_3\cdots\times_m T_z,~~~~z\in Y.$$
 It is evident that $\widetilde{\mathcal{A}}^z$ belongs to $\mathcal{A}^I$, and all $\widetilde{\mathcal{A}}^z$ tensors are symmetric when the interval tensor $\mathcal{A}^I$ itself is symmetric.
\begin{theorem}
Let $\mathcal{A}^I=[\underline{\mathcal{A}},\overline{\mathcal{A}}]$ be a symmetric interval tensor in $T_{m,n}$. Then $\mathcal{A}^I$ is Hurwitz stable if and only if for each $z\in Y$
$\widetilde{\mathcal{A}}^z$ is Hurwitz stable .
\end{theorem}
\begin{proof}
The necessity is obvious because $\widetilde{\mathcal{A}}^z\in \mathcal{A}^I$ for each $z\in Y$.

To prove the sufficiency, let $z\in Y$. Because $\widetilde{\mathcal{A}}^z$ is Hurwitz stable and symmetric, all Z-eigenvalues of $\widetilde{\mathcal{A}}^z$ are negative, therefore, all Z-eigenvalues of the symmetric tensor
$$-\widetilde{\mathcal{A}}^z =-\mathcal{A}^c- \Delta\times_1 T_z\times_2 T_z\times_3\cdots\times_m T_z$$
are positive, and according Lemma \ref{HP} $-\widetilde{\mathcal{A}}^z$ is positive definite. However, $-\widetilde{\mathcal{A}}^z$ corresponds exactly to the tensor $ \mathcal{A}^z $ associated with the interval tensor $ -\mathcal{A}^I $. Therefore, on the basis of assertion (c) of Theorem \ref{th2}, $-\mathcal{A}^I$ is positive definite. By virtue of Lemma \ref{sta}, it further follows that $ \mathcal{A}^I $ is Hurwitz stable.
\end{proof}

Theorem \ref{th2} established that the positive (semi-)definiteness of a general interval tensor can be equivalently characterized via its associated symmetric interval tensor  $\mathcal{A}_s^I$. Regrettably, this favorable property does not extend to Hurwitz stability here, only one direction of the implication holds true.

\begin{theorem}
If $\mathcal{A}_s^I$ is Hurwitz stable, then $\mathcal{A}^I$ is also Hurwitz stable.
\end{theorem}
\begin{proof}
Let $\lambda$ be an E-eigenvalue of a tensor $\mathcal{A}\in \mathcal{A}^I$. Then by the inequation \eqref{Bth} we have Re$\lambda\leq \lambda_{Z~\max}(\mathcal{A}_s) < 0$ because the symmetric tensor $\mathcal{A}_s$ belongs to $\mathcal{A}^I_s$ and thus has all E-eigenvalues negative. This proves that $\mathcal{A}^I$ is Hurwitz stable.
\end{proof}

The converse does not hold, and one may consider the example given in \cite{rohn1994} regarding interval matrix, i.e., the case where  $m=2$.

\vspace{.3cm}

\section{Some positive definite $4$th order $3$-dimensional interval tensors}
\setcounter{section}{5}
Let $\underline{\mathcal{A}}=(\underline{a}_{ijkl}), \overline{\mathcal{A}}=(\overline{a}_{ijkl})\in S_{4,3}$, and ${\mathcal{A}}^I=[\underline{\mathcal{A}},\overline{\mathcal{A}]}$, where $i,j,k,l\in[3]$. In this section we only to consider the symmetric interval tensor ${\mathcal{A}}^I$. From the results of the previous section, it follows that to prove the positive (semi-)definiteness of the $4$th order $3$-dimensional interval tensor  ${\mathcal{A}}^I$, it suffices to verify the positive (semi-)definiteness of four specific tensors $\mathcal{A}^{z_1}$, $\mathcal{A}^{z_2}$, $\mathcal{A}^{z_3}$ and $\mathcal{A}^{z_4}$, where $z_1=(1,1,1)^{\top},z_2=(1,1,-1)^{\top},z_3=(1,-1,1)^{\top},z_4=(-1,1,1)^{\top}\in Y=\{z\in \mathbb{R}^3; |z_j|=1 \mbox{ for } j=1,2,3\}$. By leveraging this insight, we have derived several sufficient conditions for the positive (semi-)definiteness of $4$th order $3$-dimensional interval tensors in this section.
Furthermore, we can further obtain some sufficient conditions for the positive (semi-)definiteness of ${\mathcal{A}}=(a_{i_1i_2i_3i_4})\in S_{4,3}$.

For tensor ${\mathcal{A}}=(a_{ijkl})\in S_{4,3}$, its corresponding homogeneous polynomial is
\begin{align*}
  f_{\mathcal{A}}(x) =& a_{1111}x_1^4+a_{2222}x_2^4+a_{3333}x_3^4+4(a_{1112}x_1^3x_2+a_{1222}x_1x_2^3+a_{2223}x_2^3x_3+a_{2333}x_2x_3^3\\
  &+a_{1113}x_1^3x_3+a_{2333}x_2x_3^3)+12(a_{1123}x_1^2x_2x_3+a_{1223}x_1x_2^2x_3+a_{1233}x_1x_2x_3^2)\\
  &+6(a_{1122}x_1^2x_2^2+a_{1133}x_1^2x_3^2+a_{2233}x_2^2x_3^2)
\end{align*}
where $x=(x_1,x_2,x_3)^{\top}\in\mathbb{R}^3$. In the subsequent discussion, we only need to focus on the scenario where $\overline{a}_{iiii}=\underline{a}_{iiii}$ and $\overline{a}_{iijj}=\underline{a}_{iijj}$ i.e., $\delta_{iiii}=\delta_{iijj}=0$ for all $i,j\in[3]$, $i\neq j$. This is because the entries in question here all correspond to the coefficients of the square terms.

\begin{theorem}\label{th3.1}
For a $4$th order $3$-dimensional real interval tensor ${\mathcal{A}}^I=[\underline{\mathcal{A}},\overline{\mathcal{A}]}$, where  $\underline{a}_{1111}=\underline{a}_{2222}=\underline{a}_{3333}=0$. Then ${\mathcal{A}}^I$ is positive semi-definite if $\underline{a}_{iiij}=\overline{a}_{iiij}=0$ for all $i,j\in[3]$, $i\neq j$, $\underline{a}_{1123}=-1$, $\overline{a}_{1123}=1$, $\underline{a}_{1223}=\underline{a}_{1233}=\overline{a}_{1223}=\overline{a}_{1233}=0$, $\underline{a}_{1122},\underline{a}_{1133}\geq1$ and $\underline{a}_{2233}\geq0$.
\end{theorem}
\begin{proof}
Let $\overline{a}_{iiii}=\underline{a}_{iiii}=0$ for all $i\in[3]$, $\overline{a}_{1122}=\overline{a}_{1133}=\underline{a}_{1122}=\underline{a}_{1133}=1$ and $\overline{a}_{2233}=\underline{a}_{2233}=0$. Then $a^c_{iiii}=a^c_{iiij}=0$ for all $i,j\in[3]$, $i\neq j$, $a^c_{1122}=a^c_{1133}=1$, $a^c_{2233}=a^c_{1123}=a^c_{1223}=a^c_{1233}=0$, $\delta_{iiii}=\delta_{iijj}=\delta_{iiij}=\delta_{1233}=\delta_{1123}=0$ for all $i,j\in[3]$ and $\delta_{1123}=1$.

Since for $x=(x_1,x_2,x_3)^{\top}\in\mathbb{R}^3\setminus\{0\}$
\begin{align*}
  f_{\mathcal{A}^{z_1}}(x)=f_{\mathcal{A}^{z_4}}(x)&=-12x_1^2x_2x_3+6(x_1^2x_2^2+x_1^2x_3^2)\\
  &=6(x_1x_2-x_1x_3)^2\\
  &\geq0,
\end{align*}
\begin{align*}
  f_{\mathcal{A}^{z_2}}(x)=f_{\mathcal{A}^{z_3}}(x)&=12x_1^2x_2x_3+6(x_1^2x_2^2+x_1^2x_3^2)\\
  &=6(x_1x_2+x_1x_3)^2\\
  &\geq0,
\end{align*}
$\mathcal{A}^z$ is positive semi-definite for each $z\in Y$. According Theorem \ref{th2}, $\mathcal{A}^I$ is positive semi-definite.
\end{proof}


\begin{corollary}\label{cor2}
Let ${\mathcal{A}}=(a_{ijkl})\in S_{4,3}$ and $a_{iiii}\geq0$, for all $i\in\{1,2,3\}$. Then $\mathcal{A}$ is positive semi-definite if $a_{iiij}=a_{1223}=a_{1233}=0$, $a_{iijj}\geq0$ for all $i,j\in\{1,2,3\}$, $i\neq j$, $a_{1122}\geq |a_{1123}|$ and $a_{1133}\geq |a_{1123}|$. Moreover, $a_{iiii}>0$, for all $i\in\{1,2,3\}$, then $\mathcal{A}$ is positive definite if $a_{iiij}=a_{1223}=a_{1233}=0$, $a_{iijj}\geq0$ for all $i,j\in\{1,2,3\}$, $i\neq j$, $a_{1122}\geq |a_{1123}|$ and $a_{1133}\geq |a_{1123}|$.
\end{corollary}

\begin{theorem}\label{th3.2}
For a $4$th order $3$-dimensional real interval tensor ${\mathcal{A}}^I=[\underline{\mathcal{A}},\overline{\mathcal{A}]}$, where  $\underline{a}_{1111}=\underline{a}_{2222}=0$ and $\underline{a}_{3333}=1$. Then ${\mathcal{A}}^I$ is positive semi-definite if $\underline{a}_{1112}=\underline{a}_{1113}=\underline{a}_{1222}=\underline{a}_{2223}=\underline{a}_{1333}=\overline{a}_{1112}=\overline{a}_{1113}=\overline{a}_{1222}=\overline{a}_{2223}=\overline{a}_{1333}=0$, $\underline{a}_{2333}=-1$, $\overline{a}_{2333}=1$ and one of the following conditions is satisfied.
\begin{itemize}
  \item [(a)]  $\underline{a}_{iijk}=\overline{a}_{iijk}=0$ for all $i,j,k\in[3]$, $i\neq j$, $i\neq k$, $j\neq k$, $\underline{a}_{1122},\underline{a}_{1133}\geq0$ and $\underline{a}_{2233}\geq\frac{2}{3}$.
  \item [(b)] $\underline{a}_{1123}=-1$, $\overline{a}_{1123}=1$, $\underline{a}_{1223}=\underline{a}_{1233}=\overline{a}_{1223}=\overline{a}_{1233}=0$, and $\underline{a}_{1122},\underline{a}_{1133}\geq1$ and $\underline{a}_{2233}\geq\frac{2}{3}$.
\end{itemize}
\end{theorem}
\begin{proof}
Let $\overline{a}_{1111}=\underline{a}_{1111}=\overline{a}_{2222}=\underline{a}_{2222}=0$, and $\overline{a}_{3333}=\underline{a}_{3333}=1$.

(a) Take $\overline{a}_{1122}=\overline{a}_{1133}=\underline{a}_{1122}=\underline{a}_{1133}=0$ and $\overline{a}_{2233}=\underline{a}_{2233}=\frac{2}{3}$. Then $a^c_{1111}=a^c_{2222}=a^c_{1122}=a^c_{1133}=a^c_{iiij}=a^c_{iijk}=0$, for all $i,j,k\in[3]$, $i\neq j$, $i\neq k$, $j\neq k$, $a^c_{3333}=1$, $a^c_{2233}=\frac{2}{3}$, $\delta_{iiii}=\delta_{iijj}=\delta_{iijk}=\delta_{1112}=\delta_{1222}=\delta_{2223}=\delta_{1333}=\delta_{1113}=0$ for all $i,j,k\in[3]$, $i\neq j$, $i\neq k$, $j\neq k$ and $\delta_{2333}=1$.

Since for $x=(x_1,x_2,x_3)^{\top}\in\mathbb{R}^3\setminus\{0\}$
\begin{align*}
  f_{\mathcal{A}^{z_1}}(x)=f_{\mathcal{A}^{z_4}}(x) &=x_3^4-4x_2x_3^3+4x_2^2x_3^2\\
  &=(2x_2x_3-x_3^2)^2\\
  &\geq0,
\end{align*}
\begin{align*}
  f_{\mathcal{A}^{z_2}}(x)=f_{\mathcal{A}^{z_3}}(x) &=x_3^4-4x_2x_3^3+4x_2^2x_3^2\\
  &=(2x_2x_3-x_3^2)^2\\
  &\geq0,
\end{align*}
$\mathcal{A}^z$ is positive semi-definite for each $z\in Y$. According Theorem \ref{th2}, $\mathcal{A}^I$ is positive semi-definite.

(b) Take $\overline{a}_{1122}=\overline{a}_{1133}=\underline{a}_{1122}=\underline{a}_{1133}=1$ and $\overline{a}_{2233}=\underline{a}_{2233}=\frac{2}{3}$. Then $a^c_{1111}=a^c_{2222}=a^c_{iiij}=a^c_{iijk}=0$ for all $i,j,k\in[3]$, $i\neq j$, $i\neq k$, $j\neq k$, $a^c_{3333}=a^c_{1122}=a^c_{1133}=1$, $a^c_{2233}=\frac{2}{3}$, $\delta_{iiii}=\delta_{iijj}=\delta_{1123}=\delta_{1223}=\delta_{1233}=\delta_{1222}=\delta_{2223}=\delta_{1333}=\delta_{1113}=0$ for all $i,j\in[3]$, $i\neq j$, and $\delta_{1123}=\delta_{2333}=1$.

Since for $x=(x_1,x_2,x_3)^{\top}\in\mathbb{R}^3\setminus\{0\}$
\begin{align*}
  f_{\mathcal{A}^{z_1}}(x)=f_{\mathcal{A}^{z_4}}(x) &=x_3^4-4x_2x_3^3-12x_1^2x_2x_3+6\times(\frac{2}{3}x_1^2x_2^2+x_2^2x_3^2+x_1^2x_3^2)\\
  &=(2x_2x_3-x_3^2)^2+6(x_1x_2-x_1x_3)^2\\
  &\geq0,
\end{align*}
\begin{align*}
  f_{\mathcal{A}^{z_2}}(x)=f_{\mathcal{A}^{z_3}}(x) &=x_3^4+4x_2x_3^3+12x_1^2x_2x_3+6\times(\frac{2}{3}x_1^2x_2^2+x_2^2x_3^2+x_1^2x_3^2)\\
  &=(2x_2x_3+x_3^2)^2+6(x_1x_2+x_1x_3)^2\\
  &\geq0,
\end{align*}
$\mathcal{A}^z$ is positive semi-definite for each $z\in Y$. According Theorem \ref{th2}, $\mathcal{A}^I$ is positive semi-definite.
\end{proof}



\begin{corollary}\label{cor3}
Let ${\mathcal{A}}=(a_{ijkl})\in S_{4,3}$, where $a_{1111},a_{2222}\geq0$ and $a_{3333}\geq1$. Then $\mathcal{A}$ is positive semi-definite if $a_{1112}=a_{1222}=a_{2223}=a_{1333}=a_{1113}=0$, $a_{iijj}\geq0$ for $i,j\in[3]$, $i\neq j$, $a_{1223}=a_{1233}=0$, $a_{3333}\geq |a_{2333}|$, $a_{2233}\geq \frac{2}{3}|a_{2333}|$ and one of the following conditions is satisfied.
\begin{itemize}
  \item [(a)] $a_{1123}=0$.
  \item [(b)] $a_{1122}\geq |a_{2333}|$, $a_{1133}\geq |a_{2333}|$, $a_{3333}\geq |a_{1123}|$, $a_{1122}\geq |a_{1123}|$, $a_{1133}\geq |a_{1123}|$, and $a_{2233}\geq \frac{2}{3}|a_{1123}|$.
\end{itemize}
 Moreover, $a_{1111},a_{2222}>0$ and $a_{3333}\geq1$, then $\mathcal{A}$ is positive definite if (a) or (b) is satisfied.
\end{corollary}

\begin{theorem}\label{th3.3}
For a $4$th order $3$-dimensional real interval tensor ${\mathcal{A}}^I=[\underline{\mathcal{A}},\overline{\mathcal{A}]}$, where  $\underline{a}_{1111}=0$ and $\underline{a}_{2222}=\underline{a}_{3333}=1$. Then ${\mathcal{A}}^I$ is positive semi-definite if $\underline{a}_{1112}=\underline{a}_{1113}=\overline{a}_{1112}=\overline{a}_{1113}=0$, and one of the following conditions is satisfied.
\begin{itemize}
  \item [(a)] $\underline{a}_{2223}=\underline{a}_{2333}=\underline{a}_{1123}=-1$, $\overline{a}_{2223}=\overline{a}_{2333}=\overline{a}_{1123}=1$, $\underline{a}_{1222}=\underline{a}_{1333}=\overline{a}_{1222}=\overline{a}_{1333}=0$, $\underline{a}_{1223}=\underline{a}_{1233}=\overline{a}_{1223}=\overline{a}_{1233}=0$, and $\underline{a}_{iijj}\geq1$ for all $i,j\in[3]$, $i\neq j$.
  \item [(b)] $\underline{a}_{2333}=-1$, $\overline{a}_{2333}=1$, $\underline{a}_{1222}=\underline{a}_{2223}=\underline{a}_{1333}=\overline{a}_{1222}=\overline{a}_{2223}=\overline{a}_{1333}=0$, $\underline{a}_{1223}=-1$, $\overline{a}_{1223}=1$, $\underline{a}_{1123}=\underline{a}_{1233}=\overline{a}_{1123}=\overline{a}_{1233}=0$, $\underline{a}_{1133},\underline{a}_{2233}\geq1$ and $\underline{a}_{1122}\geq\frac{2}{3}$.
\end{itemize}
\end{theorem}
\begin{proof}
Let $\overline{a}_{1111}=\underline{a}_{1111}=0$, and $\overline{a}_{2222}=\underline{a}_{2222}=\overline{a}_{3333}=\underline{a}_{3333}=1$.

(a) Take $\overline{a}_{iijj}=\underline{a}_{iijj}=1$ for all $i,j\in[3]$, $i\neq j$. Then $a^c_{1111}=a^c_{iiij}=a^c_{iijk}=0$, $a^c_{2222}=a^c_{3333}=a_{iijj}=1$, for all $i,j,k\in[3]$, $i\neq j$, $i\neq k$, $j\neq k$, $\delta_{iiii}=\delta_{iijj}=\delta_{1112}=\delta_{1222}=\delta_{1333}=\delta_{1113}=\delta_{1223}=\delta_{1233}=0$ for all $i,j\in[3]$, $i\neq j$ and $\delta_{2223}=\delta_{2333}=\delta_{1123}=1$.

Since for $x=(x_1,x_2,x_3)^{\top}\in\mathbb{R}^3\setminus\{0\}$
\begin{align*}
  f_{\mathcal{A}^{z_1}}(x)=f_{\mathcal{A}^{z_4}}(x) &=x_2^4+x_3^4-4(x_2^3x_3+x_2x_3^3)-12x_1^2x_2x_3+6(x_1^2x_2^2+x_2^2x_3^2+x_1^2x_3^2)\\
  &=(x_2^2-2x_2x_3+x_3^2)^2+6(x_1x_2-x_1x_3)^2\\
  &\geq0,
\end{align*}
\begin{align*}
  f_{\mathcal{A}^{z_2}}(x)=f_{\mathcal{A}^{z_3}}(x) &=x_2^4+x_3^4+4(x_2^3x_3+x_2x_3^3)+12x_1^2x_2x_3+6(x_1^2x_2^2+x_2^2x_3^2+x_1^2x_3^2)\\
  &=(x_2^2+2x_2x_3+x_3^2)^2+6(x_1x_2+x_1x_3)^2\\
  &\geq0,
\end{align*}
$\mathcal{A}^z$ is positive semi-definite for each $z\in Y$. According Theorem \ref{th2}, $\mathcal{A}^I$ is positive semi-definite.

(b) Take $\overline{a}_{1133}= \overline{a}_{2233}=\underline{a}_{1133}=\underline{a}_{2233}=1$ and  $\overline{a}_{1122}=\underline{a}_{1122}=\frac{2}{3}$ for all $i,j\in[3]$, $i\neq j$. Then $a^c_{1111}=a^c_{iiij}=a^c_{iijk}=0$, $a^c_{2222}=a^c_{3333}=a_{iijj}=1$, for all $i,j,k\in[3]$, $i\neq j$, $i\neq k$, $j\neq k$, $\delta_{iiii}=\delta_{iijj}=\delta_{1112}=\delta_{1222}=\delta_{2223}=\delta_{1333}=\delta_{1113}=\delta_{1123}=\delta_{1233}=0$ for all $i,j\in[3]$, $i\neq j$ and $\delta_{2333}=\delta_{1223}=1$.

Since for $x=(x_1,x_2,x_3)^{\top}\in\mathbb{R}^3\setminus\{0\}$
\begin{align*}
  f_{\mathcal{A}^{z_1}}(x) &=x_2^4+x_3^4-4x_2x_3^3-12x_1x_2^2x_3+6\times(\frac{2}{3}x_1^2x_2^2+x_2^2x_3^2+x_1^2x_3^2)\\
  &=(x_2^2-2x_1x_3)^2+(x_3^2-2x_2x_3+2x_1x_2)^2+2(x_1x_3-x_2x_3)^2\\
  &\geq0,
\end{align*}
\begin{align*}
  f_{\mathcal{A}^{z_2}}(x)&=x_2^4+x_3^4+4x_2x_3^3+12x_1x_2^2x_3+6\times(\frac{2}{3}x_1^2x_2^2+x_2^2x_3^2+x_1^2x_3^2)\\
  &=(x_2^2+2x_1x_3)^2+(x_3^2+2x_2x_3+2x_1x_2)^2+2(x_1x_3-x_2x_3)^2\\
  &\geq0,
\end{align*}
\begin{align*}
  f_{\mathcal{A}^{z_3}}(x)&=x_2^4+x_3^4+4x_2x_3^3-12x_1x_2^2x_3+6\times(\frac{2}{3}x_1^2x_2^2+x_2^2x_3^2+x_1^2x_3^2)\\
  &=(x_2^2-2x_1x_3)^2+(x_3^2+2x_2x_3-2x_1x_2)^2+2(x_1x_3+x_2x_3)^2\\
  &\geq0,
\end{align*}
\begin{align*}
 f_{\mathcal{A}^{z_4}}(x) &=x_2^4+x_3^4-4x_2x_3^3+12x_1x_2^2x_3+6\times(\frac{2}{3}x_1^2x_2^2+x_2^2x_3^2+x_1^2x_3^2)\\
  &=(x_2^2+2x_1x_3)^2+(x_3^2-2x_2x_3-2x_1x_2)^2+2(x_1x_3+x_2x_3)^2\\
  &\geq0,
\end{align*}
$\mathcal{A}^z$ is positive semi-definite for each $z\in Y$. According Theorem \ref{th2}, $\mathcal{A}^I$ is positive semi-definite.

\end{proof}


\begin{corollary}\label{cor4}
Let ${\mathcal{A}}=(a_{ijkl})\in S_{4,3}$, where $a_{1111}\geq0$ and $a_{2222},a_{3333}\geq1$. Then $\mathcal{A}$ is positive semi-definite if $a_{1112}=a_{1222}=a_{1333}=a_{1113}=0$, and one of the following conditions is satisfied.
\begin{itemize}
  \item [(a)] $a_{iijj}\geq|a_{2333}|$, $a_{iijj}\geq|a_{2223}|$, $a_{iijj}\geq|a_{1123}|$, $a_{iijj}\geq|a_{1123}|$ for $i,j\in[3]$, $i\neq j$, and $a_{1223}=a_{1233}=0$.
   \item [(b)] $a_{1122}\geq\frac{2}{3}|a_{2333}|$, $a_{1122}\geq\frac{2}{3}|a_{1223}|$, $a_{1133}\geq|a_{2333}|$, $a_{1133}\geq|a_{1223}|$, $a_{2233}\geq|a_{2333}|$, $a_{2233}\geq|a_{1223}|$, and $a_{1123}=a_{1233}=a_{2223}=0$.
 \end{itemize}
 Moreover, $a_{1111}>0$ and $a_{2222},a_{3333}\geq1$, then $\mathcal{A}$ is positive definite if $a_{1112}=a_{1222}=a_{1333}=a_{1113}=0$, and (a) or (b) is satisfied.
\end{corollary}

\begin{theorem}\label{th3.4}
For a $4$th order $3$-dimensional real interval tensor ${\mathcal{A}}^I=[\underline{\mathcal{A}},\overline{\mathcal{A}]}$, where  $\underline{a}_{1111}=\underline{a}_{2222}=\underline{a}_{3333}=1$. Then ${\mathcal{A}}^I$ is positive definite if $\underline{a}_{1222}=\underline{a}_{2333}=\underline{a}_{1113}=\overline{a}_{1222}=\overline{a}_{2333}=\overline{a}_{1113}=0$, $\underline{a}_{1112}=\underline{a}_{2223}=-1$, $\overline{a}_{1112}=\overline{a}_{2223}=1$ and one of the following conditions is satisfied.
\begin{itemize}
  \item [(a)] $\underline{a}_{iijj}=\frac{2}{3}$, $\underline{a}_{iijk}=\overline{a}_{iijk}=0$ for all $i,j,k\in[3]$, $i\neq j$, $i\neq k$, $j\neq k$, and $\underline{a}_{1333}=-1$, $\overline{a}_{1333}=1$.
   \item [(b)] $\underline{a}_{iijj}=1$ for all $i,j\in[3]$, $i\neq j$, $\underline{a}_{1123}=\underline{a}_{1223}=\overline{a}_{1123}=\overline{a}_{1223}=0$, $\underline{a}_{1233}=-1$, $\overline{a}_{1233}=1$, and $\underline{a}_{1333}=\overline{a}_{1333}=0$.
 \end{itemize}
\end{theorem}
\begin{proof}
Let $\overline{a}_{iiii}=\underline{a}_{iiii}=1$ for all $i\in[3]$.

(a) Take $\overline{a}_{iijj}=\underline{a}_{iijj}=\frac{2}{3}$ for all $i,j\in[3]$, $i\neq j$.
Then $a^c_{iiij}=a^c_{iijk}=0$, $a^c_{iiii}=1$, $a_{iijj}=\frac{2}{3}$ for all $i,j,k\in[3]$, $i\neq j$, $i\neq k$, $j\neq k$, $\delta_{iiii}=\delta_{iijj}=\delta_{1222}=\delta_{2333}=\delta_{1113}=\delta_{iijk}=0$ for all $i,j\in[3]$, $i\neq j$ and $\delta_{1112}=\delta_{2223}=\delta_{1333}=1$.

Since for $x=(x_1,x_2,x_3)^{\top}\in\mathbb{R}^3\setminus\{0\}$
\begin{align*}
  f_{\mathcal{A}^{z_1}}(x) &=x_1^4+x_2^4+x_3^4-4(x_1^3x_2+x_2^3x_3+x_1x_3^3)+4(x_1^2x_2^2+x_2^2x_3^2+x_1^2x_3^2)\\
  &=(x_1^2-2x_1x_2)^2+(x_2^2-2x_2x_3)^2+(x_3^2-2x_1x_3)^2\\
  &>0,
\end{align*}
\begin{align*}
  f_{\mathcal{A}^{z_2}}(x)&=x_1^4+x_2^4+x_3^4-4(x_1^3x_2-x_2^3x_3-x_1x_3^3)+4(x_1^2x_2^2+x_2^2x_3^2+x_1^2x_3^2)\\
  &=(x_1^2-2x_1x_2)^2+(x_2^2+2x_2x_3)^2+(x_3^2+2x_1x_3)^2\\
  &>0,
\end{align*}
\begin{align*}
  f_{\mathcal{A}^{z_3}}(x)&=x_1^4+x_2^4+x_3^4-4(-x_1^3x_2-x_2^3x_3+x_1x_3^3)+4(x_1^2x_2^2+x_2^2x_3^2+x_1^2x_3^2)\\
  &=(x_1^2+2x_1x_2)^2+(x_2^2+2x_2x_3)^2+(x_3^2-2x_1x_3)^2\\
  &>0,
\end{align*}
\begin{align*}
 f_{\mathcal{A}^{z_4}}(x)&=x_1^4+x_2^4+x_3^4-4(-x_1^3x_2+x_2^3x_3-x_1x_3^3)+4(x_1^2x_2^2+x_2^2x_3^2+x_1^2x_3^2)\\
  &=(x_1^2+2x_1x_2)^2+(x_2^2-2x_2x_3)^2+(x_3^2+2x_1x_3)^2\\
  &>0,
\end{align*}
$\mathcal{A}^z$ is positive definite for each $z\in Y$. According Theorem \ref{th2}, $\mathcal{A}^I$ is positive definite.

(b) Take $\overline{a}_{iijj}=\underline{a}_{iijj}=1$ for all $i,j\in[3]$, $i\neq j$.
Then $a^c_{iiij}=a^c_{iijk}=0$, $a^c_{iiii}=a_{iijj}=1$ for all $i,j,k\in[3]$, $i\neq j$, $i\neq k$, $j\neq k$, $\delta_{iiii}=\delta_{iijj}=\delta_{1222}=\delta_{2333}=\delta_{1113}=\delta_{1333}=\delta_{1123}=\delta_{1223}=0$ for all $i,j\in[3]$, $i\neq j$ and $\delta_{1112}=\delta_{2223}=\delta_{1233}=1$.

Since for $x=(x_1,x_2,x_3)^{\top}\in\mathbb{R}^3\setminus\{0\}$
\begin{align*}
  f_{\mathcal{A}^{z_1}}(x) &=x_1^4+x_2^4+x_3^4-4(x_1^3x_2+x_2^3x_3)-12x_1x_2x_3^2+6(x_1^2x_2^2+x_2^2x_3^2+x_1^2x_3^2)\\
  &=(x_1^2-2x_1x_2+x_3^2)^2+(x_2^2-2x_2x_3+2x_1x_3)^2+2(x_2x_3-x_1x_2)^2\\
  &>0,
\end{align*}
\begin{align*}
  f_{\mathcal{A}^{z_2}}(x) &=x_1^4+x_2^4+x_3^4-4(x_1^3x_2-x_2^3x_3)-12x_1x_2x_3^2+6(x_1^2x_2^2+x_2^2x_3^2+x_1^2x_3^2)\\
  &=(x_1^2-2x_1x_2+x_3^2)^2+(x_2^2+2x_2x_3-2x_1x_3)^2+2(x_2x_3+x_1x_2)^2\\
  &>0,
\end{align*}
\begin{align*}
  f_{\mathcal{A}^{z_3}}(x) &=x_1^4+x_2^4+x_3^4+4(x_1^3x_2+x_2^3x_3)+12x_1x_2x_3^2+6(x_1^2x_2^2+x_2^2x_3^2+x_1^2x_3^2)\\
  &=(x_1^2+2x_1x_2+x_3^2)^2+(x_2^2+2x_2x_3+2x_1x_3)^2+2(x_2x_3-x_1x_2)^2\\
  &>0,
\end{align*}
\begin{align*}
 f_{\mathcal{A}^{z_4}}(x) &=x_1^4+x_2^4+x_3^4-4(-x_1^3x_2+x_2^3x_3)+12x_1x_2x_3^2+6(x_1^2x_2^2+x_2^2x_3^2+x_1^2x_3^2)\\
  &=(x_1^2+2x_1x_2+x_3^2)^2+(x_2^2-2x_2x_3-2x_1x_3)^2+2(x_2x_3+x_1x_2)^2\\
  &>0,
\end{align*}
$\mathcal{A}^z$ is positive definite for each $z\in Y$. According Theorem \ref{th2}, $\mathcal{A}^I$ is positive definite.

\end{proof}
\begin{corollary}\label{cor5}
Let ${\mathcal{A}}=(a_{ijkl})\in S_{4,3}$, where $a_{1111},a_{2222},a_{3333}\geq1$. Then $\mathcal{A}$ is positive definite if $a_{1222}=a_{2333}=a_{1113}=0$, and one of the following conditions is satisfied.
\begin{itemize}
  \item [(a)] $a_{iiii}\geq|a_{1112}|$, $a_{iiii}\geq |a_{2223}|$, $a_{iiii}\geq|a_{1333}|$, $a_{iijj}\geq \frac{2}{3}|a_{1112}|$, $a_{iijj}\geq \frac{2}{3}|a_{2223}|$, $a_{iijj}\geq \frac{2}{3}|a_{1333}|$, and $a_{iijk}=0$ for $i,j,k\in[3]$, $i\neq j$, $i\neq k$, $j\neq k$.
   \item [(b)]$a_{iiii}\geq|a_{1112}|$, $a_{iiii}\geq |a_{2223}|$, $a_{iiii}\geq|a_{1123}|$, $a_{iijj}\geq |a_{1112}|$, $a_{iijj}\geq |a_{2223}|$, $a_{iijj}\geq |a_{1123}|$ for $i,j\in[3]$, $i\neq j$ and $a_{1113}=a_{1223}=a_{1233}=0$.
 \end{itemize}
\end{corollary}

\section{Conclusions}
This paper studies the positive definiteness and Hurwitz stability of interval tensors, clarifying that the positive definiteness of an interval tensor is equivalent to that of its symmetric interval tensor and auxiliary tensors $\mathcal{A}^z$. It also reveals that the stability of a symmetric interval tensor implies the stability of the interval tensor, and the stability of the symmetric interval tensor is equivalent to that of auxiliary tensors $\tilde{\mathcal{A}}^z$. In addition, sufficient conditions for the positive definiteness of 4th-order 3-dimensional interval tensors are provided.

\end{document}